\documentclass[11pt]{amsart}
\usepackage{amssymb,comment,amsmath}
\usepackage{bbm}
\usepackage{graphicx} 
\usepackage[english]{babel} 
\usepackage{float}
\usepackage{mathtools}
\usepackage{hyperref}
\usepackage[noadjust]{cite}
\usepackage[margin=3.2cm]{geometry}
\hypersetup{ 
    colorlinks=true,       
    linkcolor=blue,          
    citecolor=blue,        
    filecolor=blue,      
    urlcolor=blue           
}

\usepackage{array}

\usepackage{pstricks} 
\usepackage{pst-all}
\usepackage{epsfig}
\usepackage{pst-grad} 
\usepackage{pst-plot} 
\usepackage[space]{grffile} 
\usepackage{etoolbox} 
\usepackage{color} 
\makeatletter 
\patchcmd\Gread@eps{\@inputcheck#1 }{\@inputcheck"#1"\relax}{}{}
\makeatother

\usepackage[all]{xy}
\usepackage{mathrsfs}
\usepackage{graphics}
\usepackage{amsthm}

\theoremstyle{plain}\newtheorem{theorem}{Theorem}[section]\newtheorem{Theorem}{Theorem}\newtheorem{Corollary}{Corollary}\newtheorem{proposition}[theorem]{Proposition}\newtheorem{lemma}[theorem]{Lemma}

\theoremstyle{definition}\newtheorem{example}[theorem]{Example}\newtheorem{notation}[theorem]{Notation}\newtheorem{remark}[theorem]{Remark}

\def\Q{\mathbb{Q}}\def\C{\mathbb{C}}\def\Z{\mathbb{Z}}\def\ZZ2{\mathbb{\Z/ 2\Z}}
\def\sb{\subset}\def\ot{\otimes}\def\t{\times}\def\sm{\setminus}
\def\a{\alpha}\def\b{\beta}\def\s{\sigma}\def\De{\Delta}\def\la{\lambda}\def\la{\lambda}\def\p{\partial}\def\S{\Sigma}
\def\ov{\overline}\def\wt{\widetilde}
\def\gl{\mathfrak{gl}}
\def\sl2{\mathfrak{sl}_2}
\def\su2{\mathfrak{su}(2)}
\def\tr{\text{tr}\,}
\def\id{\text{id}}\def\End{\text{End}}

\def\deg{\text{deg}}

\def\Usl2{U_q(\sl2)}\def\usl2{\wt{U}_q(\sl2)}

\def\kk{\mathbb{K}}

\def\CC{\mathcal{C}}

\def\ZZ{\mathcal{Z}}

\def\CC{\mathcal{C}}

\def\Z{\mathbb{Z}}
\def\C{\mathbb{C}}
\def\ot{\otimes}\def\id{\text{id}}

\def\sl{\mathfrak{sl}}

\def\CC{\mathcal{C}}

\def\sl{\mathfrak{sl}_2}

\def\top{\mathrm{top}}

\def\KK{\mathbb{K}}

\def\Uqs{U_q(\mathfrak{gl}(2|1))}

\title{A plumbing-multiplicative function from the Links-Gould invariant}
\author{Daniel López Neumann and Roland van der Veen}
\email{dlopezn@udec.cl, r.i.van.der.veen@rug.nl}

\begin{document}

\maketitle

\begin{abstract}

We prove that the Laurent polynomial in $\Z[q^{\pm 1}]$ that is the top coefficient of the Links-Gould invariant of the boundary of a Seifert surface is multiplicative under plumbing of surfaces. We deduce that the Links-Gould invariant of a fibred link in $S^3$ is $\Z[q^{\pm 1}]$-monic. As a purely topological application, we deduce a ``plumbing-uniqueness'' statement for links that bound surfaces obtained by plumbing/deplumbing unknotted twisted annuli as well as providing an obstruction for links to bound such surfaces.


\end{abstract}

\section{Introduction}


\def\ttop{\text{top coefficient of }}

Quantum knot invariants are quite powerful and have many intriguing algebraic properties, but it remains a challenge to apply them to answer actual topological questions. In this paper, we make progress in this direction by applying the Links-Gould invariant of links to questions on Seifert surfaces. 
\medskip

The Links-Gould invariant of a link $L\sb S^3$ is a two-variable Laurent polynomial $LG(L;p,q)\in \Z[q^{\pm 1},p^{\pm 1}]$ \cite{LG:two-variable}. It can be considered as a $q$-deformation of the Alexander polynomial, as $LG(L;ix,i)=\De_L(x^2)$ \cite{Ishii:LG-as-generalization}. It belongs to the class of ``non-semisimple'' quantum link invariants \cite{GPT:modified, CGP:non-semisimple}, in the sense that it is a special case of the Reshetikhin-Turaev construction \cite{RT1} applied to a non-semisimple braided category: the category of modules over $\Uqs$ with {\em generic $q$}. It is the simplest link polynomial in such class with generic $q$, other non-semisimple invariants \cite{ADO, CGP:non-semisimple} require $q$ to be a root of unity. 
\medskip

Very recently, it was shown that the Links-Gould invariant satisfies a genus bound \cite{KT:Links-Gould}. This result belongs to a line of recent results stating that geometric properties of the Alexander polynomial generalize to non-semisimple quantum invariants, see \cite{LNV:genus, LNV:genus-unrolled, LNV:fibred-sl2}. In this paper, we consider the well-known fact that the top coefficient of the Alexander polynomial is multiplicative under plumbing of surfaces (see Section \ref{section: plumbing} for the definition of plumbing) and that such coefficient is $\pm 1$ for fibred links. We prove that both properties generalize to the Links-Gould invariant and we give some topological applications.
\medskip

Let $\top(L,q)\in \Z[q^{\pm 1}]$ be the coefficient of $p^{1-\chi(\S)}$ in $LG(L;p,q)$, where $\S$ is a minimal genus Seifert surface for the link $L$. Note that no power of $p$ higher than $p^{1-\chi(\S)}$ can appear in $LG(L;p,q)$ by Theorem \ref{theorem: genus bound for LG} (a generalization of \cite{KT:Links-Gould}), so that the name $\top(L,q)$ is appropriate. This top coefficient is allowed to be zero. Our main theorem is the following:

\begin{Theorem}
\label{theorem: main2, plumbing multiplicativity}
 If $\S$ is a surface obtained by plumbing surfaces $\S_1$ and $\S_2$ then $$\top(\p\S,q)=\top(\p\S_1,q)\cdot \top(\p\S_2,q).$$
\end{Theorem}

In the remainder of this introduction we will describe various corollaries of this theorem. The first deals with the form of the Links-Gould invariant of fibred links while the other corollaries are purely topological in nature. Let $L\sb S^3$ be a fibred link with fiber surface $F$. The plane field tangent to the fibers can be extended to a plane field $\xi_F$ on $S^3$ (see \cite{GG:fibred} or \cite{LNV:fibred-sl2}) and its homotopy class is determined by its {\em Hopf invariant} $\la(\xi_F)\in\Z$. Note that $\deg_pLG(L;p,q)=2(1-\chi(F))$ follows from the results of \cite{KT:Links-Gould, Ishii:LG-as-generalization}. 

\begin{Corollary}
\label{theorem: main}  
\label{Corollary: main fibred}
    If $L$ is fibred with fiber $F$, then 
    $$\top(L,q) =q^{1-\chi(F)-2\la(\xi_F)}.$$
   In particular, $LG(L;p,q)$ detects the Hopf invariant of $\xi_F$.
   \end{Corollary}

That the top coefficient of $LG(L;p,q)$ of a fibred link was a power of $q$ had been conjectured by Kohli in 2016 \cite{Kohli:LG} (without the explicit expression in terms of the Hopf invariant). 
Corollary \ref{Corollary: main fibred} is inspired by a similar theorem for Akutsu-Deguchi-Ohtsuki (ADO) invariants of fibred links \cite{LNV:fibred-sl2}. 
We deduce Corollary \ref{Corollary: main fibred} from Theorem \ref{theorem: main2, plumbing multiplicativity} using a deep theorem of Giroux-Goodman \cite{GG:fibred}: any fiber surface in $S^3$ is obtained by plumbing/deplumbing unknotted annuli with a full $\pm 1$-twist. The proof of Theorem \ref{theorem: main2, plumbing multiplicativity} relies solely on the genus bound of \cite{KT:Links-Gould} and the cubic skein relation of the Links-Gould invariant \cite{Ishii:algebraic}. Note that the Kauffman polynomial (a ``semisimple'' invariant) also has a cubic skein relation, but lacks a genus bound, hence it cannot have properties as those of $LG(L;p,q)$ shown in the present paper. 
\medskip

\subsection{Topological applications}
In light of the aforementioned theorem of Giroux-Goodman, we decided to study the class of links that bound surfaces obtained by plumbing/deplumbing unknotted annuli with an arbitrary non-zero\footnote{The non-zero twists condition guarantees that such surfaces are genus-minimizing \cite{Gabai:Murasugi1}.} number of full twists, which we call {\em annular} surfaces. If $A_n$ is an unknotted annulus with $-n$ full twists (see Figure \ref{fig.An}), we show that $\top(\p A_n,q)=a_n(q)$ where
$$a_n(q)=(2n-1)q+(2n-3)q^3+(2n-5)q^5+\dots+3q^{2n-3}+q^{2n-1}
$$ for $n> 0$ and $a_n(q)=a_{-n}(q^{-1})$ for $n<0$.
Together with Theorem \ref{theorem: main} this implies:

\begin{Corollary}
\label{Corollary: obstruction to bound ANNULAR surface}
    If $L$ bounds an annular surface then, up to a power of $q$, $\top(L,q)$ is a product of $a_n(q)$'s.
\end{Corollary}

It seems particularly sensible to ask whether positive knots bound annular surfaces since {\em all} their minimal genus Seifert surfaces have free fundamental group by \cite{Ozawa:incompressible-positive} (note that $\pi_1(S^3\sm\S)$ is free for annular $\S$). We show that many of them do not bound annular surfaces with the above criterion, see Section \ref{section: computations} for positive knots up to twelve crossings. Since positive knots are strongly quasi-positive (SQP) \cite{Rudolp:positive-are-SQ}, this also shows that there is no obvious analogue of the Giroux-Goodman theorem for SQP links (see Section \ref{section: additional comments}). Note that our criterion gives no information at $q=1$, indeed, the Alexander polynomial seems to give no obstruction to this problem.

\medskip

We also obtain the following uniqueness result for links that bound annular surfaces.

\begin{Corollary}[Plumbing-uniqueness]
    
\label{corollary: plumbing-uniqueness for iterated hopf bands}

If a link bounds an annular surface obtained by plumbing the annuli $A_{n_1},\dots, A_{n_k}$ and a surface obtained by plumbing $A_{n'_1},\dots, A_{n'_r}$, then $k=r$ and, after reordering, $n_i=n'_i$ for all $i=1,\dots,k$. More generally, suppose $L$ is a link that bounds an annular surface $\S$ obtained by plumbing $A_{n_1},\dots,A_{n_k}$ and deplumbing $A_{m_1},\dots, A_{m_{\ell}}$: 
    \begin{enumerate}

  \item If $m_1,\dots,m_{i-1}\in \{\pm 1\}$ and $m_i,\dots,m_l\neq \pm 1$, then $m_i,\dots,m_l$ is a subsequence of $n_1,\dots,n_k$ (up to reordering).
     
    \item If $L$ bounds another annular surface $\S'$ that is obtained by plumbing the annuli $A_{n'_1},\dots,A_{n'_r}$  and deplumbing $A_{m'_1},\dots,A_{m'_s}$ where $m'_1,\dots,m'_{j-1}\in \{\pm 1\}$ and $m'_j,\dots,m'_s\neq \pm 1$, then the sequences $\{n_1,\dots,n_k\}\sm \{m_i,\dots,m_l\}$ and $\{n'_1,\dots,n'_r\}\sm \{m'_j,\dots,m'_s\}$ are equal up to reordering.


    \end{enumerate}

\end{Corollary}

This result follows from Theorem \ref{theorem: main} and the fact that the $a_n(q), n\neq 0$ are pairwise coprime (up to a power of $q$), a non-obvious fact which we prove using an old theorem of Kakeya \cite{Kakeya:roots}. Note that when $n_1,\dots,n_k$ are pairwise coprime, the Alexander polynomial criterion (i.e. our main theorem at $q=1$) is enough to deduce this result. Corollary \ref{corollary: plumbing-uniqueness for iterated hopf bands} can be restated by saying that the links $\p A_n,n\neq 0$ are linearly independent in a certain Grothendieck group, see Remark \ref{remark: Grothendieck group}.
\medskip

The above uniqueness can be observed for certain classes of links for which all minimal genus Seifert surfaces are known. For instance, in \cite{HT:incompressible-surfaces}, a list of all minimal genus Seifert surfaces of a 2-bridge knot is given. One can observe that all of them are obtained by doing different plumbings over the same collection of annuli. This even holds for the larger class of special arborescent links, as proved in \cite{Sakuma:minimal-genus}. 
\medskip

\medskip

Note that (the hat version of) the top knot Floer homology group also has a plumbing-multiplicativity property \cite{CHS:Murasugi-sum-extremal-Floer}. The Floer homology result gives a criterion to obstruct a link to belong to the class of links obtained as Murasugi sums of alternating links. However, for any link $L$ in this class (e.g. alternating or positive links), the main result of \cite{CHS:Murasugi-sum-extremal-Floer} cannot obstruct $L$ from being a plumbing of annuli, as in Corollary \ref{Corollary: obstruction to bound ANNULAR surface}. This is because of the well-known fact that knot Floer homology of alternating links is determined by the Alexander polynomial and the signature \cite{OS:alternating-knots}. It is even less clear whether Corollary \ref{corollary: plumbing-uniqueness for iterated hopf bands} can be obtained with Floer homology techniques.

\medskip

We expect Theorem \ref{theorem: main} to generalize to Murasugi sum (as is the case for the Alexander polynomial and knot Floer homology). We also expect similar results to hold for other non-semisimple quantum invariants, such as those of Akutsu-Deguchi-Ohtsuki (ADO) \cite{ADO}. However, in that case, the top coefficient would be an element of a cyclotomic ring instead of a Laurent polynomial ring. This seems to be a slight disadvantage, for instance, the $p$-th ADO invariant of a fibred link only recovers the Hopf invariant of $\xi_F$ mod $2p$ \cite{LNV:fibred-sl2}. In this respect, quantum supergroups seem to provide neater results than the usual quantum groups.

\noindent {\bf Plan of the paper.} In Section \ref{section: plumbing} we define plumbing and twisted annuli. In Section \ref{section: Links-Gould} we briefly recall the definition and some properties of the Links-Gould invariant, and we (slightly) generalize the main result of \cite{KT:Links-Gould}. In Section \ref{section: proof of plumbing multiplicativity} we prove Theorem \ref{theorem: main} and Corollary \ref{Corollary: main fibred}. In Section \ref{section: proof of topological applications} we prove our corollaries for links that bound annular surfaces. In Section \ref{section: computations} we list some computations of $\top(L,q)$. Finally, in Section \ref{section: additional comments} we give a discussion on strongly quasi-positive links that motivated the corollaries above on annular surfaces.
\medskip

\noindent {\bf Acknowledgements.} We thank Ben-Michael Kohli, Chuck Livingston, Hector Peña Pollastri for interesting conversations, Matthew Hedden for some comments on their work \cite{CHS:Murasugi-sum-extremal-Floer} and Sebastian Baader for answering our questions on strongly quasi-positive links. We thank DeepSeek AI for suggesting Kakeya's theorem to prove Lemma \ref{lemma: coprime an's} under the prompt ``What can be said about the distribution of the roots of the polynomials $\top(\p A_n,q)$''. This unblocked two months of unsuccessful attempts at showing the coprimality of these polynomials. 

\section{Plumbing surfaces} 
\label{section: plumbing}

Let $\S,\S'\sb S^3$ be two compact, oriented, embedded surfaces. We suppose there exists an open 3-ball $B\sb S^3$ such that $\S'\sb B$ and $\S\sb S^3\sm \ov{B}$. Let $\s\sb \S$ and $\s'\sb \S'$ be two properly embedded arcs, one on each surface. Then we can form a new compact oriented surface $\S_p$ by gluing a square neighborhood of $\s$ in $\S$ to a square neighborhood of $\s'$ in $\S'$ in a unique way such that the arcs become perpendicular after gluing (see Figure \ref{fig.Plumbing}). We further suppose that this gluing happens in the boundary of the $3$-ball that separates the two surfaces in such a way that the orientations of the two surfaces extend to an orientation of $\S_p$. We say that $\S_p$ is obtained by {\em plumbing} the surfaces $\S$ and $\S'$. Connected sum of knots can be realized as a particularly simple plumbing of their Seifert surfaces but much more interesting plumbings are possible. The opposite operation will be referred to as \emph{deplumbing}.

\begin{figure}[htp!]
\begin{center}
\includegraphics[width=10cm]{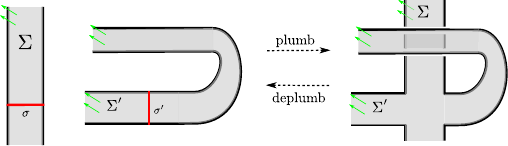}    
\end{center}
\caption{Plumbing surfaces $\S$ and $\S'$ along arcs $\s,\s'$.}
\label{fig.Plumbing}
\end{figure}

\subsection{Twisted unknotted annuli}
\label{subs: An's} For any $n\in\Z\sm \{0\}$, let $A_n$ be an unknotted annulus embedded in $S^3$ with $-n$ full twists (a positive twist is right-handed and a negative one is left-handed), see Figure \ref{fig.An}. The choice of sign is so that $\p A_n=H_n$ as oriented links, where $H_n$ is the oriented $(2,n)$-torus link whose components have linking number $n$. We call $A_n$ a {\em twisted annulus}, and $A_{\pm 1}$ a {\em Hopf band} (positive or negative). It is well-known that $H_n$ is a fibered link with fiber $A_n$ if and only if $n=\pm 1$. 
\medskip

\begin{figure}[htp!]
\begin{center}
\includegraphics[width=7cm]{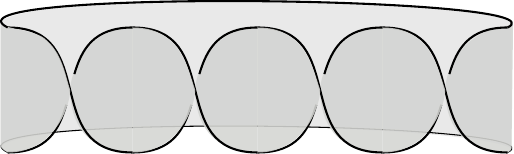}    
\end{center}
\caption{The unknotted twisted annulus $A_{-2}$.}
\label{fig.An}
\end{figure}

\section{Links-Gould invariants}

\label{section: Links-Gould}

\subsection{The Links-Gould invariant} 
\label{subs: Links-Gould invariant}
We assume the reader is familiar with the Reshetikhin-Turaev invariants of $\CC$-colored, framed, oriented tangles, where $\CC$ is a ribbon tensor category \cite{RT1, Turaev:BOOK1}. To define the Links-Gould invariant, we let $\CC$ be the category of weight super-representations of the quantum supergroup $\Uqs$ where $q$ is generic (see \cite{AGP:relative-modular-superalgebras} and \cite{dWLK:On-Links-Gould} for more details on $\Uqs$). This is a ribbon tensor category \cite{Khoroshkin-Tolstoi:universal-R-matrix-superalgebras, Yamane:quantized-superalgebras-R-matrices}. The only fact we will need from the representation theory of $\Uqs$ is that there exists a family of (4-dimensional) simple highest weight modules $V_{\a}$ in $\CC$ with highest weight $\a\in \C$ satisfying the following properties\footnote{This family is a quantized version of a similar one-parameter family of simple modules of classical $\gl(2|1)$, see Kac \cite{Kac:Lie-superalgebras}.}:
\begin{enumerate}
    \item[(a)] The braiding and pivotal structure of $\CC$ act on $V_{\a}$ with coefficients (in a preferred PBW-type basis of $V_{\a}$) that are Laurent polynomials in $q$ and $q^{\a}$ independent of $\a$ and with integer coefficients. In other words, each such coefficient has the form $f(q,q^{\a})$ for some polynomial $f(q,t)\in \Z[q^{\pm 1}, t^{\pm 1}]$ independent of $\a$. See \cite[Section 4]{dWLK:On-Links-Gould} or \cite{Ishii:LG-as-generalization} for the formula of the braiding.
   \item[(b)] The ribbon element acts as $1$ on $V_{\a}$.\footnote{This is why none of the works \cite{dWLK:On-Links-Gould, Ishii:LG-as-generalization, Ishii:algebraic} include a framing normalization in the definition of $LG(L;p,q)$. We checked this in Mathematica using the formula of the $R$-matrix of \cite{dWLK:On-Links-Gould}.}
    \item[(c)] The quantum dimension of each $V_{\a}$ is zero.
    
     \item[(d)] $\End_{\CC}(V_{\a}\ot V_{\a})$ is 3-dimensional, where $\End_{\CC}$ denotes endomorphisms of $\Uqs$-modules.\footnote{This follows from the fact that $V_{\a}\ot V_{\a}$ decomposes as a sum of three pairwise non-isomorphic simple modules, see \cite[Section 4.6.1]{deWit:thesis}.}
 \end{enumerate}
 \medskip

We denote by $RT_{V_{\a}}(T)$ the Reshetikhin-Turaev invariant of a tangle $T$ in which all components are colored with $V_{\a}$. Note that, by $(b)$, $RT_{V_{\a}}(T)$ is actually an invariant of unframed oriented tangles, hence from now on we only consider unframed oriented tangles and links. If $T$ is an oriented $(1,1)$-tangle, since $V_{\a}$ is simple $RT_{V_{\a}}(T):V_{\a}\to V_{\a}$ has the form $c_T\cdot \id_{V_{\a}}$ and by $(a)$ $c_T$ has the form $f_T(q,q^{\a})$ for some polynomial $f(q,t)\in \Z[q^{\pm 1}, t^{\pm 1}]$. It turns out that $c_T$ is a Laurent polynomial in $q$ and $q^{2\a+1}$ (see \cite{Ishii:LG-as-generalization}, in a slightly different convention for the variables). From now on, we set $p:=q^{2\a+1}$ and consider $c_T$ as an element of $\Z[q^{\pm 1}, p^{\pm 1}]$.
\medskip

Since the quantum dimension vanishes by property (c) above we define the Links-Gould invariant on oriented links $L$ by choosing a $(1,1)$-tangle $L_o$ whose closure is $L$. It is known that $c_{L_o}\in\Z[q^{\pm 1},p^{\pm 1}]$ is independent of the $(1,1)$-tangle $L_o$ chosen \cite{AGP:relative-modular-superalgebras}, hence $$LG(L;p,q):=c_{L_o}$$
is a link invariant. This is the Links-Gould invariant\footnote{In many references \cite{Ishii:algebraic, dWLK:On-Links-Gould}, the Links-Gould invariant is a polynomial $LG(L;t_0,t_1)$. The invariant of the present paper is obtained by setting $t_0=p^{-1}q, t_1=pq$. Also, some of these works use the variable $p=q^{\a+\frac{1}{2}}$ but we prefer to set $p=q^{2\a+1}$.}. 
\medskip

\subsection{Basic properties of LG} The Links-Gould invariant has the following simple properties:
\begin{itemize}
    \item It is symmetric under $p\mapsto p^{- 1}$ and thus can be written as
\begin{align}
    \label{eq: LG symmetric version in p}
    LG(L;p,q)=\sum_{k=0}^Nf_k(q)(p^{k}+p^{-k})
\end{align}
for some $N\geq 0$ and $f_k(q)\in\Z[q^{\pm 1}]$, see \cite{dWLK:On-Links-Gould}.
\medskip

\item If $\ov{L}$ is the mirror image of $L$, then $LG(\ov{L};p,q)=\ov{LG(L;p,q)}$ where $f\mapsto \ov{f}$ is the unique $\Z$-linear involution of $\Z[q^{\pm 1},p^{\pm 1}]$ that maps $q\mapsto q^{-1}, p\mapsto p^{-1}$ \cite{dWLK:On-Links-Gould}.
\medskip

\item   At $q=1$, $LG(L;p,q)$ recovers the Alexander polynomial by $$LG(L;p,1)=\De_L(p)^2,$$
   and at $q=-1$ one has $$LG(L;-p,-1)=\De_L(p)^2,$$
   see \cite{Kohli:LG-square-Alexander}. Moreover, at $q=i$  $$LG(L;ip,i)=\De_L(p^2),$$
    see \cite{Ishii:LG-as-generalization}.
\end{itemize}

The following sections deal with more elaborate properties of the Links-Gould invariant.

\subsection{Skein relations} In this section we explore the consequences of property $(d)$ above. In what follows, all our tangle diagrams are assumed to be colored with $V_{\a}$ and they are read from bottom to top.

\begin{lemma}
\label{lemma: 3-dimensional End and basis}
 We have $\dim\End_{\CC}(V_{\a}^*\ot V_{\a})=3$. Moreover, the following $V_{\a}$-colored tangles 
    
\begin{figure}[H]
\begin{align*}
\begin{matrix}
\includegraphics[width=1.6cm]{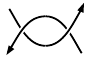}
\end{matrix}, \hspace{1.5cm}
\begin{matrix}
\includegraphics[width=1cm]{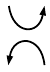}
\end{matrix}, \hspace{1.5cm}
\begin{matrix}
\includegraphics[width=1cm]{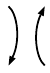}
\end{matrix}
\end{align*}
\end{figure}
form a basis of $\End_{\CC}(V_{\a}^*\ot V_{\a})$.
\end{lemma}

\begin{proof}
The first assertion follows since $\End_{\CC}(V_{\a}^*\ot V_{\a})$ is isomorphic to $\End_{\CC}(V_{\a}\ot V_{\a})$ (just reverse the downward arrows using the ribbon structure of $\CC$). Let $T_1,T_2,T_3$ be the above tangles, from left to right, and suppose there is a linear dependence relation of the form 
\begin{align}
        aT_1+bT_2+cT_3=0.
        \label{eq: linear independence of 3 tangles}
\end{align}
Then, closing the two leftmost strands (i.e. applying a left partial trace) and using properties $(b),(c)$ of Subsection \ref{subs: Links-Gould invariant}, we obtain $(a+b)\id_{V_{\a}}=0$ so $a+b=0$. Similarly, closing the two top strands we obtain $aLG(H_+)+c=0$ where $H_+$ is the positive Hopf link. Now, composing (\ref{eq: linear independence of 3 tangles}) with a negative horizontal double braiding on the left and closing the two top strands again gives the equation $bLG(H_-)+c=0$, where $H_-$ is the negative Hopf link. Since $LG(H_+)\neq -LG(H_-)$ it is easy to see that these three equations have the unique solution $a=b=c=0$. Since $\End_{\CC}(V_{\a}^*\ot V_{\a})$ is 3-dimensional, it follows that the set $\{T_1,T_2,T_3\}$ is a basis.
\end{proof}

Sometimes we will need the explicit coefficients of a morphism in this basis. In
\cite{Ishii:algebraic} the following explicit expression is obtained:


\begin{lemma}[\cite{Ishii:algebraic}]
    \label{lemma: Ishii's skein}
    For every $n>0$ define $B_n\in\Z[q^{\pm 1}]$ and $C_n,D_n\in\Z[p^{\pm 1},q^{\pm 1}]$ by
    \begin{align*}
    B_n&=\frac{q^{2n}-1}{q^2-1}, &  C_n&=2nR-B_n((q^2+1)R+1), & D_n&=(q^2-1)R\cdot B_n+1,
\end{align*}
where $R=\frac{-q(p+p^{-1})+(q^2+1)}{q^2-1}$. Then

\begin{figure}[H]
\begin{align*}
\begin{matrix}
\includegraphics[width=0.7cm]{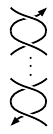}
\end{matrix} \ \Bigg\} \text{ $n$ full twists} 
=B_n \begin{matrix}
\includegraphics[width=1.6cm]{LGpdf4.pdf}
\end{matrix}
+C_n\begin{matrix}
\includegraphics[width=1cm]{LGpdf2.pdf}
\end{matrix}   
 +D_n\begin{matrix}
\includegraphics[width=1cm]{LGpdf3.pdf}
\end{matrix}.
\end{align*}
\end{figure}

\end{lemma}

\medskip

As an immediate consequence we can compute $LG(H_n) = B_n+C_n$ by closing the left side of each tangle and using properties b) and c) of the module $V_\a$ mentioned at the beginning of Subsection \ref{subs: Links-Gould invariant}.
\medskip


When $n=1$ we refer to the above formula as a ``cubic skein relation'' since it directly follows from the $R$-matrix having a cubic minimal polynomial (see \cite{Ishii:algebraic}). Note that $C_n,D_n$ are polynomials of degree one in $p+p^{-1}$ with coefficients in $\Z[q^{\pm 1}]$, while $B_n$ does not depend on $p$ at all.
\medskip

It will be useful to have an expression of $\id_{V_{\a}^*\ot V_{\a}}$ in terms of horizontal double braidings. We can obtain this from the cubic skein relation (Lemma \ref{lemma: Ishii's skein}) as follows. When $n=2$, rotating the tangles counterclockwise by 90 degrees and then rotating them along a horizontal axis by 180 degrees results in

\begin{figure}[H]
\begin{align*}
\begin{matrix}
\includegraphics[width=2cm]{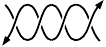}
\end{matrix} 
=B_2 \begin{matrix}
\includegraphics[width=0.7cm]{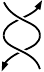}
\end{matrix}
+C_2\begin{matrix}
\includegraphics[width=0.7cm]{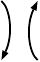}
\end{matrix}   
 +D_2\begin{matrix}
\includegraphics[width=0.7cm]{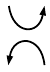}
\end{matrix}.
\end{align*}
\end{figure}

Now, multiply the cubic skein relation ($n=1$) by $B_2=1+q^2$:

\begin{figure}[H]
\begin{align*}
B_2\begin{matrix}
\includegraphics[width=0.7cm]{1LGvertical.pdf}
\end{matrix} 
=B_2B_1 \begin{matrix}
\includegraphics[width=1.5cm]{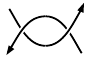}
\end{matrix}
+B_2C_1\begin{matrix}
\includegraphics[width=0.7cm]{LGcapcup.pdf}
\end{matrix}   
 +B_2D_1\begin{matrix}
\includegraphics[width=0.7cm]{LGidentity.pdf}
\end{matrix}.
\end{align*}
\end{figure}

Subtracting this from the previous equation, we can cancel the vertical double braiding:

\begin{figure}[H]
\begin{align*}
\begin{matrix}
\includegraphics[width=2cm]{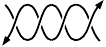}
\end{matrix} 
=B_2B_1 \begin{matrix}
\includegraphics[width=1.5cm]{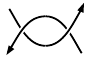}
\end{matrix}
+(D_2+B_2C_1)\begin{matrix}
\includegraphics[width=0.7cm]{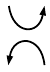}
\end{matrix}   
 +(C_2+B_2D_1)\begin{matrix}
\includegraphics[width=0.7cm]{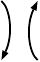}
\end{matrix}.
\end{align*}
\end{figure}

A simple computation shows that $$D_2+B_2C_1=-q^2, \hspace{1cm} C_2+B_2D_1=-2(1+q^2)+2q(p+p^{-1}).$$
Denote $g(q,p)=q^{-1}(C_2+B_2D_1)$, note that $g(q,p)$ has degree one in $p+p^{-1}$. Multiplying by $q^{-1}$, the previous equation can be written as

\begin{figure}[H]
\begin{align}
\label{eq: identity in terms of horizontal braidings}
g(q,p)\begin{matrix}
\includegraphics[width=0.7cm]{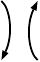}
\end{matrix} 
=q \begin{matrix}
\includegraphics[width=0.7cm]{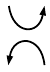}
\end{matrix}
-(q^{-1}+q)\begin{matrix}
\includegraphics[width=1.5cm]{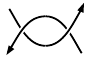}
\end{matrix}   
 +q^{-1}\begin{matrix}
\includegraphics[width=2cm]{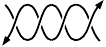}
\end{matrix}.
\end{align}
\end{figure}

\subsection{Genus bounds} For a polynomial $f(p,q)\in \Z[q^{\pm 1},p^{\pm 1}]$ written as 
\[f(p,q)=\sum_{k=r}^sf_k(q)p^k \text{ with } f_k(q)\in \Z[q^{\pm 1}],\ f_s,f_r\neq 0,\ r\leq s\] 
we use the notation $\deg_pf:=s-r$ and $\deg_p^+:=s$. Note that $\deg_p^+$ satisfies $\deg_p^+(f+g)\leq \max\{\deg_p^+f,\deg_p^+g\}$ for any $f,g\in\Z[q^{\pm 1},p^{\pm 1}]$.

\begin{theorem}
\label{theorem: genus bound for LG}
     Let $\S$ be a compact oriented surface embedded in $S^3$ with no closed components. Then $$\deg_p LG(\p \S; p,q)\leq 2( 1-\chi(\S)).$$
\end{theorem}

By the symmetry under $p\mapsto p^{-1}$ this is equivalent to $\deg_p^+LG(\p \S)\leq 1-\chi(\S)$. Note that in \cite{KT:Links-Gould} this result is proved for connected $\S$ with connected boundary. The argument there applies to general connected $\S$ as well: this is because the argument of \cite{KT:Links-Gould} only uses that the surface is obtained by attaching 1-handles (say $l$ 1-handles) to a single disk, and proves that $\deg_p LG(\p \S)\leq 2l$, which is the above theorem for connected $\S$. If $\S$ is disconnected and formed from $k$ disks by attaching $l$ 1-handles, the above theorem gives the improved bound $\deg_p LG(\p\S)\leq 2(1-k+l)$ instead of merely $\deg_p LG(\p\S)\leq 2l$. We deduce this from the connected case and the cubic skein relation.

\begin{proof}[Proof Theorem \ref{theorem: genus bound for LG}]
We proceed by induction on the number $k$ of components of $\S$. If $k=1$, this is the result of \cite{KT:Links-Gould}. Let $\S$ be a surface with $k+1$ components. Let $\S_0$ be obtained from $\S$ by joining two different components of $\S$ with a band. Let $\S_1$ be obtained from $\S_0$ by doing a full left-handed twist to that band and $\S_2$ be obtained by doing two full left-handed twists to the same band. Then, using (\ref{eq: identity in terms of horizontal braidings}), we get: 
\begin{equation}\label{eq:connect}g(q,p)LG(\p \S)=q LG(\p \S_0)-(q^{-1}+q)LG(\p \S_1)+q^{-1}LG(\p \S_2)\end{equation}
 Note that $\S_2,\S_1,\S_0$ have the same Euler characteristic and $\chi(\S_0)=\chi(\S)-1$. By the induction hypothesis, and since all coefficients are polynomials in $q$, the right hand side of Equation \eqref{eq:connect} is a polynomial with $\deg^+_p\leq 1-\chi(\S_0)=2-\chi(\S)$. But the left hand side has $\deg_p^+=1+\deg^+_pLG(\p \S)$ so we get $\deg_p^+LG(\p \S)\leq 1-\chi(\S)$ as desired. This finishes the induction and proves the theorem.

\end{proof}

\begin{figure}[htp!]
    \begin{center}
    \includegraphics[width=3cm]{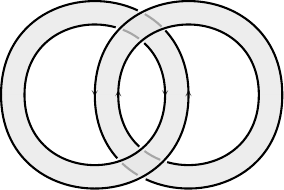}
    \caption{A Hopf link of annuli.}
    \label{fig.HopfAnnuli}
\end{center}
\end{figure}

\begin{example}

    Let $\S$ be the surface obtained from two untwisted annuli forming a Hopf link as in Figure \ref{fig.HopfAnnuli}. We computed
$$LG(\p \S,p,q) = -q^{-4}+2-q^4+ (q^3+q^{-3}-q-q^{-1})(p+p^{-1}).$$
This satisfies that $\deg_p^+LG(\p \S)\leq1-\chi(\S)=1$ as stated in the above theorem.
\end{example}

\section{Proof of the plumbing multiplicativity theorem}
\label{section: proof of plumbing multiplicativity}

In this section we prove Theorem \ref{theorem: main2, plumbing multiplicativity} and Corollary \ref{Corollary: main fibred}.
 
\begin{lemma}
\label{lemma: plumbing Hopf band}
   Theorem \ref{theorem: main2, plumbing multiplicativity} is true if one of the surfaces $\S_1$ or $\S_2$ is a Hopf band.
\end{lemma}
\begin{proof}
Suppose we plumb a positive Hopf band $A_1$ to a surface $\S_2$, call $\S$ the resulting surface. Since the plumbing happens in a ball neighborhood of a properly embedded arc on $\S_2$, we can isotope $\S$ to look like the first term in Figure \ref{fig.HopfAnnuli}. Applying Lemma \ref{lemma: Ishii's skein} for $n=1$ on the band that is plumbed gives a relation involving the following four surfaces:

\begin{figure}[htp!]
    \begin{center}
    \includegraphics[width=\linewidth]{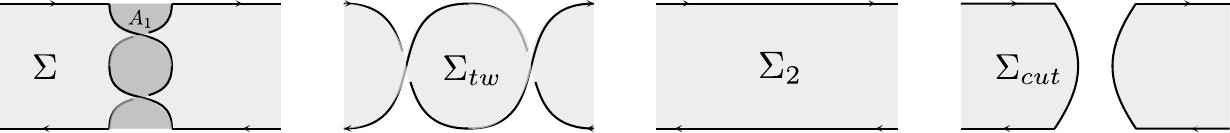}
    \caption{Left: the surface $\S$ obtained from plumbing $A_1$ onto $\S_2$. The other surfaces arise from applying the cubic skein relation.}
    \label{fig.HopfAnnuli}
\end{center}
\end{figure}

   The relation is 
   \begin{align}
   \label{eq: Hopf plumbing relation}
       LG(\p \S)=B_1 LG(\p \S_{tw})+C_1LG(\p \S_2)+D_1 LG(\p \S_{cut}).
   \end{align}
   Note that $\chi(\S_{tw})=\chi(\S_2)$ and $\chi(\S_{cut})=\chi(\S_2)+1$. By Theorem \ref{theorem: genus bound for LG} (note that $\S_{cut}$ might be disconnected) we have 
   \begin{align}
   \label{eq: rightmost term HOPF}
       \deg_p^+(D_1LG(\p \S_{cut}))\leq 1+1-\chi(\S_{cut})=1-\chi(\S_2)
   \end{align}
   since $\deg_p^+(D_1)=1$ while 
   \begin{align}
   \label{eq: leftmost term HOPF}
       \deg_p^+(B_1LG(\p \S_{tw}))\leq 1-\chi(\S_2)
   \end{align}
   since $B_1$ has degree zero. Since $\chi(\S)=\chi(\S_2)-1$ the last two equations and (\ref{eq: Hopf plumbing relation}) imply that the top term of $LG(\p\S)$ (i.e. the term corresponding to $p^{1-\chi(\S)}=p^{2-\chi(\S_2)}$) is the same as the top term of $C_1LG(\p\S_2)$. Since the coefficient of $p$ in $C_1$ is the top coefficient of $LG(\p A_1)$ (which is $q$), we get $$\top(\p\S,q)=q\top(\p\S_2,q)=\top(\p A_1,q)\top(\p\S_2,q)$$ as desired.

   The case for a negative Hopf band can be treated analogously and is left to the reader.
\end{proof}

Now let $\S_0$ be a compact oriented surface with two disjoint oriented arcs $\sigma_1,\sigma_2$ on its boundary, see Figure \ref{fig.sigma}. 
Gluing the two arcs yields a surface $\S$ while deleting the arcs $\sigma_1,\sigma_2$ from $\p \S_0$ yields a $(2,2)$-tangle we call $T_\S$. We denote the Reshetikhin-Turaev invariant (recall that all strands are assumed to be colored with $V_{\a}$) of $T_\S$ by $LG(T_{\S})\in\End_{\CC}(V_{\a}^*\ot V_{\a})$.
\begin{figure}[H]
    \centering
    \includegraphics[width=8cm]{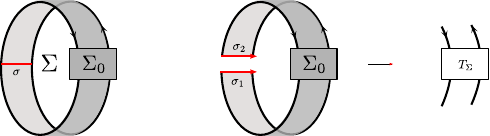}
    \caption{The surface $\S$ is obtained from gluing together the two arcs on the boundary of $\S_0$. The tangle $T_\S$ is also shown.}
    \label{fig.sigma}
\end{figure}

\begin{lemma}
\label{lemma: cut F along alpha, express in basis}
    We have 
    \begin{figure}[H]
\begin{align}
\label{eq: TSigma in a,b,c}
\begin{matrix}
\includegraphics[width=1cm]{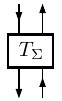}
\end{matrix}=a \begin{matrix}
\includegraphics[width=1.6cm]{LGpdf4.pdf}
\end{matrix}
+b\begin{matrix}
\includegraphics[width=1cm]{LGpdf2.pdf}
\end{matrix}   
 +c\begin{matrix}
\includegraphics[width=1cm]{LGpdf3.pdf}
\end{matrix}
\end{align}
\end{figure}
for some $a,b,c\in\Q(q^{\pm 1},p^{\pm 1})$ satisfying 
\begin{enumerate}
    \item $\deg^+_p(a)<\deg^+_p(b)$,
    \item $\deg^+_p(c)\leq \deg^+_p(b)$,
    \item $\deg^+_p(b)\leq 1-\chi(\S)$ and the coefficient of $p^{1-\chi(\S)}$ in $b$ equals that of $LG(\p \S ;p,q)$.
\end{enumerate}

\end{lemma}
\begin{proof}
Referring to the pictures in Figure \ref{fig.sigma} denote by $\S'$ the surface obtained by plumbing a negative Hopf band to $\S$ along $\s$:
\begin{figure}[H]
    \centering
    \includegraphics[width=8cm]{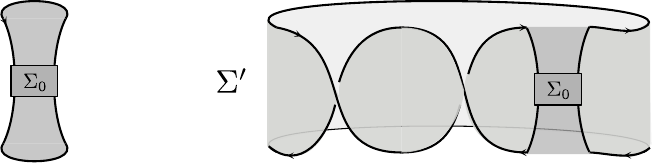}
\end{figure}

Since $V_{\a}$ is simple, we have

\begin{figure}[H]
\begin{align}
\label{eqtn TSigma, Sigma cut, Sigma prime}
\begin{matrix}
\includegraphics[width=1.6cm]{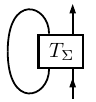}
\end{matrix} 
&=LG(\p\S)\cdot\begin{matrix}
\includegraphics[width=0.2cm]{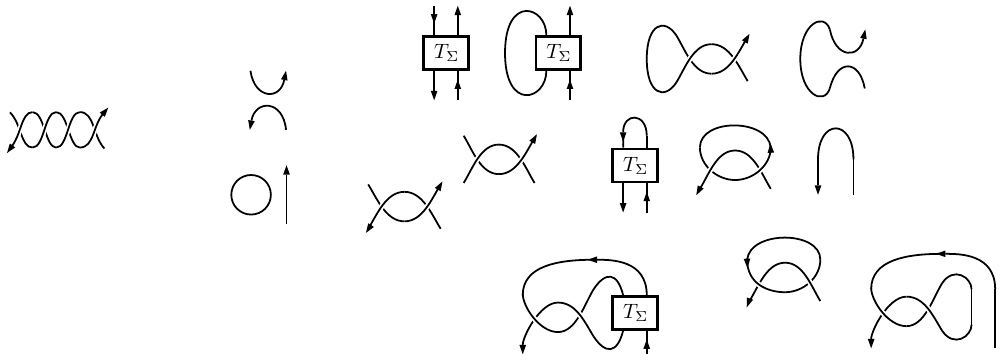}
\end{matrix} \ , &
\begin{matrix}
\includegraphics[width=1cm]{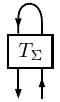}
\end{matrix}   
&=LG(\p\S_{0})\cdot\begin{matrix}
\includegraphics[width=1cm]{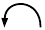}
\end{matrix} \ , \\
 \begin{matrix}
\includegraphics[width=2.3cm]{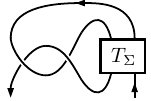}
\end{matrix}   
 &=LG(\p\S')\cdot\begin{matrix}
\includegraphics[width=1cm]{cap}
\end{matrix}  & & \nonumber
\end{align}
\end{figure}

By Lemma \ref{lemma: 3-dimensional End and basis} we can write $LG(T_{\S})$ with some coefficients $a,b,c \in \Q(q^{\pm 1},p^{\pm 1})$ as in (\ref{eq: TSigma in a,b,c}). We have to show they satisfy the listed properties. To do this, we will express $a,b,c$ in terms of Links-Gould invariants of the links $\p\S,\p\S_{0}$ and $\p\S'.$ To begin, we take the left trace in (\ref{eq: TSigma in a,b,c}):

\begin{figure}[H]
\begin{align*}
\begin{matrix}
\includegraphics[width=1.5cm]{proofLGTSigmaleftclosure.pdf}
\end{matrix} 
=a\cdot\begin{matrix}
\includegraphics[width=1.7cm]{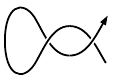}
\end{matrix}
+b\cdot\begin{matrix}
\includegraphics[width=1.4cm]{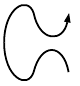}
\end{matrix}   
 +c\cdot \begin{matrix}
\includegraphics[width=1.5cm]{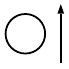}
\end{matrix}.
\end{align*}
\end{figure}

The LHS is $LG(\p\S)\cdot\id_{V_{\a}}$ by (\ref{eqtn TSigma, Sigma cut, Sigma prime}) and the RHS is $(a+b)\cdot\id_{V_{\a}}$ since $V_{\a}$ has zero quantum dimension and the twist on $V_{\a}$ is 1 (see $(b)$ and $(c)$ in Subsection \ref{subs: Links-Gould invariant}). Hence we obtain
\begin{align}
\label{eq: LG Sigma a,b}
    LG(\p \S)=a+b
\end{align}

Now close the two top strands in (\ref{eq: TSigma in a,b,c}):

\begin{figure}[H]
\begin{align*}
\begin{matrix}
\includegraphics[width=1cm]{proofLGTSigmatopclosure.pdf}
\end{matrix} 
=a\begin{matrix}
\includegraphics[width=1.3cm]{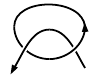}
\end{matrix}
+b\begin{matrix}
\includegraphics[width=1cm]{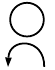}
\end{matrix}   
 +c\begin{matrix}
\includegraphics[width=1cm]{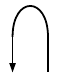}
\end{matrix}.
\end{align*}
\end{figure}

By a similar argument using (\ref{eqtn TSigma, Sigma cut, Sigma prime}) we obtain

\begin{align}
\label{eq: LG Scut a,c}
    LG(\p \S_{0})=aLG(H_+)+c.
\end{align}

Finally, compose (\ref{eq: TSigma in a,b,c}) with a negative horizontal double twist on the left:

\begin{figure}[H]
\begin{align*}
\begin{matrix}
\includegraphics[width=2cm]{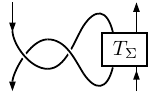}
\end{matrix} 
=a\begin{matrix}
\includegraphics[width=0.7cm]{LGcapcup.pdf}
\end{matrix}
+b\begin{matrix}
\includegraphics[width=1.4cm]{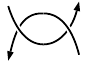}
\end{matrix}   
 +c\begin{matrix}
\includegraphics[width=1.7cm]{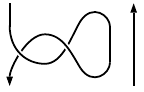}
\end{matrix}.
\end{align*}
\end{figure}

Closing the two top strands and again using (\ref{eqtn TSigma, Sigma cut, Sigma prime}) we see that
\begin{align}
\label{eq: LG Sprime b,c}
    LG(\p \S')=bLG(H_-)+c.
\end{align}

Now, multiply (\ref{eq: LG Sigma a,b}) by $LG(H_-)$, substract (\ref{eq: LG Sprime b,c}) from it and replace $c$ by $LG(\p \S_{0})-aLG(H_+)$ ((\ref{eq: LG Scut a,c}) above). This results in the following equation for $a$:
\begin{align}
    \label{eq: solving for a}
    LG(H_-)LG(\p \S)+LG(\p \S_{0})-LG(\p \S')=(LG(H_-)+LG(H_+))a.
\end{align}

By the genus bound, the LHS has $\deg_p^+\leq 2-\chi$ where $\chi=\chi(\S)$ while the RHS has degree $1+\deg_p^+(a)$ (since $\deg_p^+(LG(H_-)+LG(H_+))=1$), hence $\deg_p^+(a)\leq 1-\chi$. We will show that $\deg_p^+(a)<1-\chi$. First, note that $\deg_p^+LG(\p \S_{0})\leq -\chi$ so $\S_{0}$ does not contribute to the top term of the LHS in \ref{eq: solving for a}. On the other hand, by Lemma \ref{lemma: plumbing Hopf band}, the coefficient of $p^{2-\chi}$ (i.e. the top term) of $LG(\p \S')$ is $\top(\p  A_-,q)\top(\p\S,q)$. But this is the same top term as in $LG(H_-)LG(\p \S)$ and these two cancel up in (\ref{eq: solving for a}). Thus, the LHS has $\deg_p^+<2-\chi$ from which $\deg_p^+(a)<1-\chi$ follows. Using this, $\deg_p^+(c)\leq 1-\chi$ follows from (\ref{eq: LG Scut a,c}) since $\deg_p^+LG(\p \S_{0})\leq -\chi$ and $\deg_p^+(aLG(H_+))\leq 1-\chi$. Finally, since $LG(\p \S)=a+b$ and $\deg_p^+(a)\leq -\chi$, it follows that the coefficient of $p^{1-\chi}$ in $b$ is the same as that of $LG(\p \S)$, as desired. 
\end{proof}

\begin{proof}[Proof of Theorem \ref{theorem: main2, plumbing multiplicativity}]
Let $\S$ be the surface obtained by plumbing $\S_1$ onto $\S_2$ along an arc $\s\sb \S_2$. We can suppose the plumbing happens in a small neighborhood of the form $N(\s)\t I$ on the negative side of $\S_2$ (here $N(\s)$ denotes a tubular neighborhood of $\s\sb \S_2$). Applying Lemma \ref{lemma: cut F along alpha, express in basis} to $\S_1$, we get a relation involving the following surfaces:
    \begin{figure}[H]
        \centering
        \includegraphics[width=\linewidth]{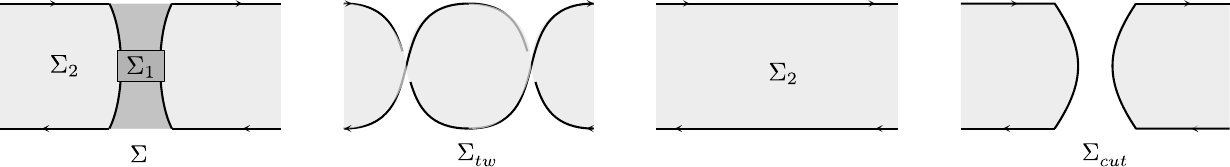}
      
    \end{figure}

    \medskip
This relation is 
\begin{align}
\label{eq: proof of THM2 relation}
    LG(\p \S)=a\cdot LG(\p \S_{tw})+b\cdot LG(\p \S_2)+c\cdot LG(\p \S_{cut})
\end{align}
where $a,b,c$ are as in the previous lemma. By the genus bound, $\deg_p^+ LG(\p \S_{cut})\leq -\chi(\S_2)$ so that by the same lemma $$\deg_p^+(c\cdot LG(\p \S_{cut}))\leq 1-\chi(\S_1)-\chi(\S_2).$$
We also have $\deg_p^+(a)\leq -\chi(
\S_1
)$ by the lemma so that $$\deg_p^+(a\cdot LG(\p \S_{tw}))\leq -\chi(\S_1)+1-\chi(\S_2).$$
Hence, only the middle term in the RHS of (\ref{eq: proof of THM2 relation}) can reach the term of degree $$1-\chi(\S)=1-(\chi(\S_1)+\chi(\S_2)-1)=2-\chi(\S_1)-\chi(\S_2).$$
Moreover, the top coefficient in $q$ of the middle term is precisely the top coefficient in $q$ of $LG(\p \S_1)LG(\p \S_2)$ by the above lemma. Hence the top coefficient in $q$ of $LG(\p \S)$ is the same as that of $LG(\p \S_1)LG(\p \S_2)$ which proves the theorem.
\end{proof}

\begin{proof}[Proof of Corollary \ref{Corollary: main fibred}]
    By a theorem of Giroux-Goodman \cite{GG:fibred}, every fiber surface in $S^3$ is obtained by plumbing/deplumbing Hopf bands $A_{\pm 1}$ (see Subsection \ref{subs: An's}). Suppose $F$ is obtained by plumbing $n_{\pm}$ copies of $A_{\pm 1}$ and deplumbing $m_{\pm}$ copies of $A_{\pm 1}$. Since $\top(\p A_{\pm 1},q)=q^{\pm 1}$, by Theorem \ref{theorem: main2, plumbing multiplicativity} we obtain $\top(L,q)=q^{n_+-n_--m_++m_-}$. The number $n_+-n_--m_++m_-$ coincides with $1-\chi(F)-2\la(\xi_F)$, see e.g. \cite{GG:fibred} or \cite{LNV:fibred-sl2}. Since $q$ is generic, $LG(L)$ detects the exponent of $q$ in $\top(L,q)$ and since $\deg_p^+ \ LG(L)=1-\chi(F)$ (because of the genus bound for Links-Gould and the fact that $LG(L)$ specializes to Alexander \cite{Ishii:LG-as-generalization}), it follows that $LG(L)$ detects the Hopf invariant as $$\la(\xi_F)=\frac{1}{2}(\deg^+_pLG(L)-\text{ exponent of $q$ in } \top(L,q)).$$
    
\end{proof}

\begin{remark}
   Note that, by the Giroux-Goodman theorem, Lemma \ref{lemma: plumbing Hopf band} was enough to deduce Corollary \ref{Corollary: main fibred}.
\end{remark}

\section{Proof of the topological applications} 
\label{section: proof of topological applications}

We now prove our corollaries for annular surfaces, i.e. surfaces obtained by plumbing and deplumbing unknotted twisted annuli. Note that we have $$LG(H_n)=C_n+B_n=(b_n(q)+B_n)+a_n(q)(p+p^{-1})$$
for some coefficients $a_n(q), b_n(q)$, see the remark below Lemma \ref{lemma: Ishii's skein}.

\begin{lemma}
    If $n>0$, the top coefficient of $LG(H_n)$ is $$a_n(q)=(2n-1)q+(2n-3)q^3+(2n-5)q^5+\dots+3q^{2n-3}+q^{2n-1}.$$
    The top coefficient of $LG(H_{-n})$ is $a_n(q^{-1})$.
\end{lemma}

\begin{proof}
Let $a_n$ be defined by the above equation. From the definitions we see that the coefficient $a'_n$ of $p$ in $C_n$ is $$a'_n=\frac{-q}{q^2-1}[2n-B_n(q^2+1)].$$
    The coefficients $a_n$ defined above satisfy the recurrence $a_{n+1}=q^2a_n+(2n+1)q$. We will show that $a'_n$ satisfy the same recurrence: indeed
    \begin{align*}
        a'_{n+1}&=\frac{-q}{q^2-1}[2n+2-(1+q^2c_1(n))(1+q^2)]\\
        &=\frac{-q}{q^2-1}[1-q^2+2n(1-q^2)+q^2[2n-c_1(n)(q^2+1)]]\\
        &=q(2n+1)+q^2a'_n.
    \end{align*}
    Since $a'_1=q=a_1$ it follows that $a'_n=a_n$ for all $n$. The second assertion follows from $LG(H_{-n};p,q)=LG(H_n;p,q^{-1})$ since $H_{-n}$ is the mirror image of $H_n$.
\end{proof}

One can also compute $b_n(q)$ explicitly as $b_n(q)= -2n-\sum_{k=0}^{n-1} (4k+1)q^{2(n-k)}$ but we are not going to need this.

\medskip
To prove Corollary \ref{corollary: plumbing-uniqueness for iterated hopf bands} we will need the following old theorem of Kakeya from 1912 \cite{Kakeya:roots}.

\begin{theorem}[Kakeya]
    Let $f(x)=a_0+a_1x+\dots a_nx^n$ be a polynomial with positive real coefficients. Then all the complex roots of $f(x)$ lie in the annulus $R_{min}(f)\leq |z|\leq R_{max}(f)$ where $$R_{min}(f)=\min_j \frac{a_j}{a_{j+1}}, \ \  R_{max}(f)=\max_j\frac{a_j}{a_{j+1}}.$$
\end{theorem}

For every $n>0$ let $d_n(q)=q^{-1}a_n(q)$ and let $d_{-n}(q)$ be defined by $a_n(q^{-1})=q^{-(2n-1)}d_{-n}(q)$, that is, $$d_{-n}(q)=(2n-1)q^{2n-2}+(2n-3)q^{2n-4}+\dots+3q^2+1.$$

\begin{lemma}
\label{lemma: coprime an's}
    The polynomials $d_n(q),n\in\Z$ are pairwise coprime in $\Z[q]$.
\end{lemma}
\begin{proof}
The statement is equivalent to proving that the polynomials have no common root in $\C$. All these are polynomials in $q^2$, so it clearly suffices to show that the polynomials $p_n(q)=d_n(q^{1/2}), n\in\Z$ have no common root. Let's prove it first for $\{p_n\}_{n>0}$, so
$$p_n(q)=d_n(q^{1/2})=(2n-1)+(2n-3)q+(2n-5)q^2+\dots+q^{n-1}.$$ If $n>m>0$, it is easy to see that $p_n(q)=f(q)+q^{n-m}p_m(q)$ where $$f(q)=(2n-1)+(2n-3)q+\dots+(2m+1)q^{n-m-1}.$$
We will show that $p_m$ and $f(q)$ cannot have a common root, this clearly implies that $p_m$ and $p_n$ do not have a common root.
Note that the sequence $\frac{2k+1}{2k-1}$ is decreasing and converges to 1 from the right. Hence $R_{max}(f)=\frac{2m+3}{2m+1}$ and by Kakeya's theorem, the roots of $f$ satisfy $|z|\leq \frac{2m+3}{2m+1}$. But for $p_m$ one has $R_{min}(p_m)=\frac{2m-1}{2m-3}$ and by the same theorem, the roots of $p_m$ satisfy $|z|\geq \frac{2m-1}{2m-3}$. Since $R_{max}(f)<R_{min}(p_m)$, it follows that $p_m$ and $f$ cannot have a common root. Since $p_{-n}(q)=q^{n-1}p_n(q^{-1})$ and $p_{-n}(0)\neq 0$, it follows that the $p_{-n},n>0$ are also pairwise coprime. Moreover, in the previous proof we also showed that , for $n>0$, all roots of $p_n$ satisfy $|z|>1$. Hence all roots of $p_{-n}$ satisfy $|z|<1$, so the $p_n$ are coprime to the $p_{-m}$ for $n,m>0$.
\end{proof}

In fact we conjecture that the $d_n(q)$'s are irreducible. This was checked  by computer in hundreds of cases but a general proof is missing.

\begin{proof}[Proof of Corollaries \ref{Corollary: obstruction to bound ANNULAR surface} and \ref{corollary: plumbing-uniqueness for iterated hopf bands}]
Suppose $L$ bounds a surface $\S$ obtained by plumbing annuli $A_{n_1},\dots,A_{n_k}$ and deplumbing $A_{m_1},\dots,A_{m_l}$ with $m_1,\dots,m_{i-1}\in \{\pm 1\}$ and $m_i,\dots,m_l\neq \pm 1$. By Theorem \ref{theorem: main}, we have $$\top(L,q)=\frac{a_{n_1}(q)\dots a_{n_k}(q)}{a_{m_1}(q)\dots a_{m_l}(q)}=q^{u}\frac{d_{n_1}(q)\dots d_{n_k}(q)}{d_{m_i}(q)\dots d_{m_l}(q)}$$
for some $u\in\Z$ (note that $a_{\pm 1}(q)=q^{\pm 1}$). Since $\top(L,q)\in\Z[q^{\pm 1}]$, $\Z[q^{\pm 1}]$ is a PID and the $d_n(q),n\in\Z$ are pairwise coprime, it follows that each $d_a(q)$ in the denominator above must appear in the numerator, that is, $m_i,\dots,m_l$ is a subsequence of $n_1,\dots,n_k$. This clearly implies Corollary \ref{Corollary: obstruction to bound ANNULAR surface}. Now, if we have another expression of $\S$ as an annular surface as stated in the corollary, then we would have $$q^{u}\frac{d_{n_1}(q)\dots d_{n_k}(q)}{d_{m_i}(q)\dots d_{m_l}(q)}=q^{u'}\frac{d_{n'_1}(q)\dots d_{n'_r}(q)}{d_{m'_{j}}(q)\dots d_{m'_{s}}(q)}.$$
Since no $d_{n}(q)$ is divisible by $q$ it follows that $u=u'$ and again by Lemma \ref{lemma: coprime an's}, it follows that the sequence $n_1,\dots,n_k$ minus $m_i,\dots,m_l$ equals the sequence $n'_1,\dots,n'_k$ minus $m'_j,\dots,m'_s$. 
\end{proof}

It is easy to decide whether a given polynomial in $\Z[q^{\pm 1}]$ is a product of $a_n(q)$'s. For instance, we have:

\begin{proposition}
\label{prop: top coeff. of annular links}
Let $L$ be a link that bounds an annular surface. Let $\a\in\Z$ be the top coefficient of $\top(L,q)$ and $\b\in\Z$ the least coefficient. Then:
    \begin{enumerate}
        \item Both $\a,\b$ are odd.
        \item Let $\a=p_1^{n_1}\dots p_k^{n_k}$ be the prime factorization of $\a$. Similarly write $\b={p'_1}^{m_1}\dots {p'_l}^{m_l}$. Then $$\deg_q\top(L,q)\geq \sum_{i=1}^kn_i(p_i-1)+\sum_{j=1}^lm_j(p'_j-1).$$
    \end{enumerate}
\end{proposition}
\begin{proof}
The first statement follows easily since all coefficients of the $a_n(q)$'s are odd. To prove the second one, note first that 
\begin{align}
\label{eq: minimum over factorization}
    \min\Bigg\{\sum_{i=1}^r(d_i-1) \ \Bigg| \  d_1,\dots,d_r\in\Z \text{ such that } \prod_{i=1}^rd_i=\a \Bigg\}=\sum_{i=1}^kn_i(p_i-1).
\end{align}
Indeed, if any $d_i$ was not prime, say $d_i=ab$ with $a,b>1$ then clearly $d_i-1>(a-1)+(b-1)$. Hence, the minimum of the above set is achieved when each $d_i$ is a prime divisor of $\a$. Now suppose $L$ is annular. By part $(1)$ of Corollary \ref{corollary: plumbing-uniqueness for iterated hopf bands}, $\top(L,q)$ has to be a product of $a_n(q)$'s (up to a power of $q$, but this does not affects $\deg_q$). Write $\top(L,q)=q^sf(q)g(q)$ where $f(q)=\prod_{i=1}^ra_{d'_i}(q^{-1})$ and $g(q)=\prod_{j=1}^sa_{c'_j}(q)$, each $d'_i,c'_j>0$ and $s\in\Z$. Write $d_i=2d'_i-1$ and $c_j=2c'_j-1$. Note that $g(q)$ is monic in $q$ with least coefficient $\prod_{j=1}^sc_j$, while $f(q)$ is monic in $q^{-1}$ with top coefficient $\prod_{i=1}^r d_i$. Hence, $\a=\prod_{i=1}^r d_i$ and $\b=\prod_{j=1}^sc_j$ while the degree of $f(q)g(q)$ is the sum $$\sum_{i=1}^r (d_i-1)+\sum_{j=1}^s(c_j-1).$$
By (\ref{eq: minimum over factorization}), the minimum degree of $f(q)g(q)$ subject to the constraint on top/least coefficient, is given by the prime factorizations of $\a$ and $\b$. 
\end{proof}

\def\GK{K^{plumb}(\text{links})}

\begin{remark}
    \label{remark: Grothendieck group}
        Corollary \ref{corollary: plumbing-uniqueness for iterated hopf bands} can be restated in terms of the following Grothendieck group: let $\GK$ be the free abelian group spanned by all isotopy classes $[L]$ of links $L\sb S^3$ with $\top(L,q)\neq 0$ (e.g. alternating, fibred, positive links) modulo the relations $[\p\S]-[\p\S_1]-[\p\S_2]$ for every compact, oriented surfaces $\S,\S_1,\S_2$ in which $\S$ is a plumbing of $\S_1,\S_2$ (and the top coefficients of the Links-Gould invariants of $\S,\S_1,\S_2$ are all non-zero). Theorem \ref{theorem: main} implies that $\top(L,q)$ defines a group homomorphism $$\GK\to\Q(q^{\pm 1}).$$
        It is easy to see that a family of links with pairwise coprime images in $\Z[q^{\pm 1}]$ are linearly independent in $\GK$. Hence, we have shown that (the images of) the iterated Hopf links $H_n=\p A_n, n\neq 0$ are linearly independent in $\GK$. This also implies Corollary \ref{corollary: plumbing-uniqueness for iterated hopf bands}: if $L$ bounds a surface obtained by plumbing $A_{n_1},\dots,A_{n_k}$ then $$[L]=[\p A_{n_1}]+\dots+[\p A_{n_k}]$$
        in $\GK$ (and similarly if deplumbing is allowed), so plumbing-uniqueness follows from linear independence of the $[H_n],n\neq 0$. As mentioned in the introduction, we conjecture that $\top(L,q)$ is multiplicative under Murasugi sum, hence we expect that the $[H_n],n\neq 0$ are linearly independent in the Grothendieck group obtained by replacing ``plumbing'' by Murasugi sum. Still, such a group would not be the same as the one of Cheng-Hedden-Sarkar \cite{CHS:Murasugi-sum-extremal-Floer}, since not all links are being considered. It is an interesting question whether the links $H_n,n\neq 0$ are linearly independent in the Grothendieck group of \cite{CHS:Murasugi-sum-extremal-Floer} too.
    \end{remark}

\section{Computations} \label{section: computations} A list with the Links-Gould invariants for prime knots with $\leq 10$ crossings can be found in \cite{DeWit:10-crossings}. Tables \ref{table: small knots} and \ref{table: 9-crossing knots} are based on those computations. For such small knots one has $\deg_p^+ LG(K)=2g(K)$. Denote $a_n(q)$ simply by $a_n$, thus $a_1=q, a_2=3q+q^3, a_3=5q+3q^3+q^5$, etc.

\begin{center}
\begin{table}[h]
    \begin{tabular}{|c|c|c|c|c|c|c|c|c|c|c|c|}
    \hline
       Knot & $3_1$ & $4_1$ & $5_1$ & $5_2$ & $6_1$ & $6_2$ & $6_3$ & $7_1$ & $7_2$ & $7_3$ & $7_4$   \\
        \hline
        Fibred & Y &Y &Y&N&N&Y&Y&Y&N&N&N\\
        \hline
        $\top(K,q)$ & $q^2$& $1$ &$q^4$ &$qa_2$ & $q^{-1}a_2$ & $q^2$ & $1$ & $q^6$ & $qa_3$ & $q^3a_2$ & $a_2^2$  \\
         \hline
         \hline
       $7_5$ & $7_6$ & $7_7$ &  $8_1$ &   $8_2$ &  $8_3$ &  $8_4$ &  $8_5$ &  $8_6$ &  $8_7$ &  $8_8$ &  $8_9$    \\
         \hline
 N& Y&Y&N&Y&N&N&Y&N&Y&N&Y\\
         \hline
    $q^3a_2$& $q^2$ & $1$  &   $q^{-1}a_3$ &  $q^4$ & $a_2a_{-2}$ & $q^3a_{-2}$ & $q^4$ & $qa_2$ & $q^2$ & $q^{-1}a_2$ & $1$    \\
         \hline
         \hline
 $8_{10}$ &  $8_{11}$& $8_{12}$ &  $8_{13}$ &  $8_{14}$  & $8_{15}$&$8_{16}$ &$8_{17}$ & $8_{18}$&$8_{19}$ &$8_{20}$ &$8_{21}$ \\
  \hline
Y&N&Y&N &N & N & Y &Y &Y &Y &Y &Y \\
\hline
 $q^2$ & $qa_2$ & $1$ & $q^{-1}a_2$ & $qa_2$& $6q^4+3q^6$ & $q^2$ &1 &1 &$q^6$ &1 &$q^2$  \\
         \hline
           
    \end{tabular}
\medskip

    \caption{$\top(K,q)$ and fibredness status of knots with $\leq 8$ crossings.}
    \label{table: small knots}
  \end{table}
\end{center}

We checked that all knots other than $8_{15}$ in Table \ref{table: small knots} do bound a surface made by plumbing annuli as specified from the top coefficient. For instance, $8_3$ (which has genus 1) bounds a plumbing of $A_2$ and $A_{-2}$, and $8_4$ bounds a plumbing of $A_{-2}$ with a certain number of Hopf bands. Since $g(8_4)=2$ there should be 3 Hopf bands (assuming no deplumbing is needed), hence the $q^3$ factor implies it must be a plumbing of three $A_1$'s, which is easy to see by applying Seifert's algorithm to a minimal (alternating) diagram of $8_4$. It has to be noted too that De Wit uses the mirror image of $8_4, 8_8, 8_{13}$ as they appear in KnotInfo. 

\medskip

Note that in this table $8_{15}$ is the only knot for which the top coefficient of $LG$ is not a product of $a_n(q)$'s. This can be observed from $(1)$ of Prop. \ref{prop: top coeff. of annular links}, or by noting that $3|\top(8_{15},q)$ and that no $a_n(q), n\in\Z\sm\{0\}$ is divisible by $3$. Note that by applying Seifert's algorithm to the standard diagram of $8_{15}$ one can see it is obtained by plumbing two positive Hopf bands to a pretzel surface, see Figure \ref{figure: pretzel piece}. The knot $9_{49}$ is also obtained by plumbing Hopf bands to such a pretzel surface, as indicated in the same figure. Hence, both the pretzel link that is the boundary of such pretzel surface and the knot $9_{49}$ cannot bound an annular surface either.

\begin{figure}
    \centering
    \includegraphics[width=3.5cm]{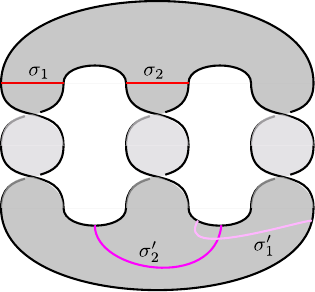}
        \caption{The knot $8_{15}$ (resp. $9_{49}$) is the boundary of the surface obtained by plumbing two positive Hopf bands on the arcs $\s_1,\s_2$ (resp. $\s'_1,\s'_2$). }
    \label{figure: pretzel piece}
\end{figure}
\medskip

In Table \ref{table: 9-crossing knots} we list $\top(K,q)$ for all 9-crossing knots. We see that the knots $9_k$ with $k=25,35,38,39,41,46,49$ do not bound any annular surface, since $\top(K,q)$ is not a product of $a_n(q)$'s. This can be checked in Mathematica or using Proposition \ref{prop: top coeff. of annular links}, for instance, $\deg_q(9_{35})=6<19-1$ so $9_{35}$ cannot bound an annular surface. For all other 9-crossing knots $\top(K,q)$ is a product of $a_n(q)$'s and one can check that the given knot bounds a plumbing of annuli as specified by $\top(K,q)$.

\medskip
\begin{table}[h]
\[
\begin{array}{c|c}
\text{{\bf top coefficient}}&\text{{\bf knot}}\\
\hline 
1 & 9_{24},9_{27},9_{30},9_{33},9_{34},9_{44},\\
\hline
q^2 & 9_{17},9_{22},9_{26},9_{28},9_{29},9_{31},9_{32},9_{40},9_{42},9_{45},9_{47},9_{48}\\
\hline
q^4 & 9_{11},9_{20},9_{36},9_{43}\\
\hline
q^8 & 9_1\\
\hline
qa_2 & 9_{12},9_{15},9_{21}\\
\hline
q^{-1}a_2 & 9_{14},9_{19},9_{37}\\
\hline
 q^5a_2  & 
9_3,9_6,9_9,9_{16}
 \\  \hline 
q^3a_{-2} & 9_8\\
\hline 
 q^3a_3 & 9_4,9_7
 \\ \hline 
  qa_4 & 9_2\\
 \hline 
a_2a_3 & 9_5
 \\ \hline 
 q^2a_2^2 & 9_{10},9_{13},9_{18},9_{23}
 \\ \hline 
  19q^2+20q^4+9q^6+q^8 & 9_{35}
 \\ \hline 
 14q^4+10q^6+q^8 & 9_{38} 
 \\ \hline 
 2+q^2+q^4 & 9_{46}\\
\hline
 6+3q^2 & 9_{41}\\
\hline
 6q^2+3q^4 & 9_{25},9_{39}\\
\hline
 6q^4+3q^6 & 9_{49} 
 \\ \hline 
\end{array}
\]
    \caption{$\top(K,q)$ of 9-crossing knots.}
    \label{table: 9-crossing knots}
  \end{table}

For higher crossing knots, we only consider the positive ones since, as mentioned in the introduction, all their minimal genus Seifert surfaces have free fundamental group \cite{Ozawa:incompressible-positive}, so it seems more subtle to decide whether they bound annular surfaces (which have free $\pi_1$) or not. In Table \ref{table: positive knots} we list all positive knots with $\leq 12$ crossings whose $\top(K,q)$ is a product of $a_n(q)$'s. For the 10-crossing knots in Table \ref{table: positive knots}, we did checked that they indeed bound an annular surface as specified by $\top(K,q)$ (only 3 of them were non-fibred $10_{128},10_{134},10_{142}$, the rest are fibred hence annular by the Giroux-Goodman theorem). We haven't checked which of the 11-crossing and 12-crossing knots of Table \ref{table: positive knots} are annular. There could be some non-annular ones, indeed, Kohli notices in \cite{Kohli:LG-Alexander} that there are non-fibred 12-crossing knots with $\Z[q^{\pm 1}]$-monic Links-Gould invariant (though none of them is positive).

\medskip

Table \ref{table: positive knots, non-annular} shows the positive knots up to 12-crossings with $\top(K,q)$ that is not a product of $a_n(q)$'s, hence none of them is annular. We checked that $\top(K,q)$ is not a product of $a_n(q)$'s with Mathematica, but it can also be seen easily with Proposition \ref{prop: top coeff. of annular links}. For instance, none of the knots with $\top(K,q)=3q^5+15q^4+18q^3$ is annular because the last coefficient is even (or just because $3|\top(K,q)$) and none of the knots with $\top(K,q)=3q^5+21q^4+25q^3$ is annular because $(3-1)+2(5-1)=10$ and $\deg_q\top(K,q)=2$.
\medskip



On the basis of the tables, we conjecture that the top coefficient of all
positive knots is a positive integer combination of positive powers of $q$.

\begin{table}
\[
\begin{array}{c|c}
\text{{\bf top coefficient}}&\text{{\bf positive knots}}\\
\hline 
 
 q^6 &  10_{154},  10_{161}, 11_{n183}
 \\ \hline 
 q^8 & 
 \begin{array}{c}
 10_{152}, 10_{124}, 10_{139}, 11_{n77}, 12_{n91}, 12_{n105},\\ 12_{n136}, 12_{n187}, 12n_{328}, 12_{n417}, 12_{n426},\\ 12_{n518}, 12_{n591}, 12_{n640},\\ 12n_{647}, 12_{n694},12_{n850}
\end{array} 
 \\ \hline 
 q^6 \left(q^2+3\right) & 
  \begin{array}{c}
  10_{128}, 10_{134}, 10_{142}, 12_{n96},12_{n110}, 12_{n217},\\ 12_{n594}, 12_{n644}, 12_{n655}, 12_{n851}, 12_{n638}
\end{array}
 \\ \hline 
 q^6 \left(q^2+3\right)^2 & 
\begin{array}{c}
 11_{a94} , 
 11_{a186} , 
 11_{a191} , 
 11_{a235} , 
 11_{a236} , 
 11_{a241} , 
 11_{a335} ,\\ 
 11_{a336} ,
 11_{a340} , 
 11_{a356} , 
 11_{a357} , 
 11_{a365} , 
 12_{n169} , 
 12_{n177} , 
 12_{n245} , \\
 12_{n289} , 
 12_{n308} , 
 12_{n341} , 
 12_{n477} , 
 12_{n503} , 
 12_{n585} , 
 12_{n600}
\end{array}
 \\ \hline 
  q^4 \left(q^2+3\right) \left(q^4+3 q^2+5\right) & 
\begin{array}{c}
  11_{a95},
  11_{a238},
  11_{a243} ,
  11_{a341} ,
  11_{a359} ,
  11_{a360}
\end{array}
 \\ \hline 
 q^4 \left(q^2+3\right)^3 & 
\begin{array}{c}
 11_{a192} ,
 11_{a237} ,
 11_{a337} ,
 11_{a366}
\end{array}
 \\ \hline 
  q^8 \left(q^2+3\right) & 
\begin{array}{c}
 11_{a234} ,
 11_{a240} ,
 11_{a263} ,
 11_{a334} ,
 11_{a338} ,
 11_{a355} ,
 11_{a364} ,
 12_{n74} ,\\
 12_{n153} ,
 12_{n166} ,
 12_{n243} ,
 12_{n244} ,
 12_{n292} ,
 12_{n305} ,
 12_{n338} ,
 12_{n374} ,\\
 12_{n386} ,
 12_{n473} ,
 12_{n474} ,
 12_{n502} ,
 12_{n575} ,
 12_{n576} ,
 12_{n680} ,\\
 12_{n689} ,
 12_{n691} ,
 12_{n692}
\end{array}
 \\ \hline 
 q^6 \left(q^4+3 q^2+5\right) & 
\begin{array}{c}
 11_{a242} ,
 11_{a245} ,
 11_{a339} ,
 11_{a358} ,
 12_{n77} ,
 12_{n251} ,
 12_{n581}
\end{array}
 \\ \hline 
 q^4 \left(q^6+3 q^4+5 q^2+7\right) & 
\begin{array}{c}
 11_{a246} ,
 11_{a342}
\end{array}
 \\ \hline 
 q^2 \left(q^8+3 q^6+5 q^4+7 q^2+9\right) & 
\begin{array}{c}
 11_{a247}
\end{array}
 \\ \hline 
  q^2 \left(q^4+3 q^2+5\right)^2 & 
\begin{array}{c}
 11_{a363} 
\end{array}\\ \hline
 q^{10} & 
\begin{array}{c}
 11_{a367} ,
 12_{n242} ,
 12_{n472} ,
 12_{n574} ,
 12_{n679} ,\\
 12_{n688} ,
 12_{n725} ,
 12_{n888} 
\end{array}
 \\ \hline 
\end{array}
\]
    \caption{Positive prime knots with 10-12 crossings such that its Links-Gould invariant has top coefficient which is a product of $a_n(q)$'s.}
    \label{table: positive knots}
  \end{table}

  \begin{table}
\[
\begin{array}{c|c}
\text{{\bf top coefficient}}&\text{{\bf positive knots}}\\
\hline 
  3 q^6 \left(q^2+2\right) & 
\begin{array}{c}10_{49}, 10_{66}, 10_{80},
 11_{n93},  11_{n126},\\ 11_{n136}, 11_{n169}, 11_{n180}, 12_{n203},12_{n764} 
\end{array}
 \\ \hline 
  q^4 \left(q^6+17 q^4+54 q^2+49\right) & 
\begin{array}{c}
 11_{a329} 
\end{array}
 \\ \hline 
 q^2 \left(q^2+3\right) \left(q^6+3 q^4+5 q^2+7\right) & 
\begin{array}{c}
 11_{a343} 
\end{array}
 \\ \hline 
 q^2 \left(q^8+9 q^6+25 q^4+36 q^2+29\right) & 
\begin{array}{c}
 11_{a362} 
\end{array}
 \\ \hline 
  3 q^4 \left(q^2+2\right) \left(q^2+3\right) & 
\begin{array}{c}
 10_{53}, 11_{n171}
\end{array}
 \\ \hline 
  q^4 \left(3 q^4+10 q^2+12\right) & 
\begin{array}{c}
 10_{55}, 10_{63}, 11_{n181}
\end{array}
 \\ \hline 
 q^4 \left(5 q^4+21 q^2+23\right) & 
\begin{array}{c}
 10_{101}
\end{array}
 \\ \hline 
 q^4 \left(5 q^4+28 q^2+31\right) & 
\begin{array}{c}
 10_{120}
\end{array}
 \\ \hline 
  2 q^6 \left(3 q^2+5\right) & 
\begin{array}{c}
 11_{a43}, 12_{n88}, 12_{n133}, 12_{n259}, 12_{n806}
\end{array}
 \\ \hline 
 q^4 \left(q^6+13 q^4+35 q^2+32\right) & 
\begin{array}{c}
  11_{a123} ,
  11_{a320} ,
  11_{a354}
\end{array}
 \\ \hline 
 q^6 \left(q^4+10 q^2+14\right) & 
\begin{array}{c}
  11_{a124} ,
  11_{a227} ,
  11_{a244} ,
  11_{a291} ,
  11_{a298} ,
  11_{a318} ,
  11_{a319} ,\\
  11_{a353} ,
  12_{n100} ,
  12_{n406} ,
  12_{n453} ,
  12_{n758}
\end{array}
 \\ \hline 
 q^4 \left(q^6+9 q^4+20 q^2+19\right) & 
\begin{array}{c}
 11_{a200} ,
 11_{a361}
\end{array}
 \\ \hline 
 q^4 \left(q^2+3\right) \left(q^4+10 q^2+14\right) & 
\begin{array}{c}
 11_{a292} 
\end{array}
 \\ \hline 
 q^4 \left(q^6+10 q^4+27 q^2+26\right) & 
\begin{array}{c}
 11_{a299} 
\end{array}\\ \hline 
 3 q^6 \left(q^2+2\right) \left(q^2+3\right) & 
\begin{array}{c}
 12_{a43} ,
 12_{a53} ,
 12_{a82} ,
 12_{a94} ,
 12_{a144} ,
 12_{a145} ,
 12_{a295} ,\\
 12_{a368} ,
 12_{a575} ,
 12_{a814} ,
 12_{a877} 
\end{array}
 \\ \hline 
 3 q^8 \left(q^2+2\right) & 
\begin{array}{c}
 12_{a52} ,
 12_{a93} ,
 12_{a143} ,
 12_{a276} ,
 12_{a367} ,
 12_{a574} ,
 12_{a647} ,\\
 12_{a811} ,
 12_{a813} ,
 12_{a817} ,
 12_{a876} 
\end{array}
 \\ \hline 
 q^6 \left(3 q^4+10 q^2+12\right) & 
\begin{array}{c}
 12_{a55} ,
 12_{a96} ,
 12_{a277} ,
 12_{a344} ,
 12_{a355} ,\\
 12_{a420} ,
 12_{a442} ,
 12_{a648} 
\end{array}
 \\ \hline 
 q^4 \left(3 q^6+10 q^4+18 q^2+18\right) & 
\begin{array}{c}
 12_{a56} \\
 12_{a679} \\
\end{array}
 \\
 q^4 \left(q^2+3\right) \left(3 q^4+10 q^2+12\right) & 
\begin{array}{c}
 12_{a97} ,
 12_{a421} ,
 12_{a443}
\end{array}
 \\ \hline 
 q^6 \left(3 q^4+21 q^2+25\right) & 
\begin{array}{c}
 12_{a102}, 
 12_{a107} 
\end{array}
 \\ \hline 
 3 q^4 \left(q^2+2\right) \left(q^4+3 q^2+5\right) & 
\begin{array}{c}
 12_{a152} 
\end{array}
 \\ \hline 
 q^4 \left(3 q^6+24 q^4+52 q^2+42\right) & 
\begin{array}{c}
 12_{a156}
\end{array}
 \\ \hline 
 q^6 \left(5 q^4+21 q^2+23\right) & 
\begin{array}{c}
 12_{a293}  ,
 12_{a319} ,
 12_{a391} ,
 12_{a432} ,
 12_{a490} ,\\
 12_{a586} ,
 12_{a615} ,
 12_{a828} ,
 12_{a1035} 
\end{array}
 \\ \hline 
 q^4 \left(5 q^6+31 q^4+61 q^2+47\right) & 
\begin{array}{c}
 12_{a320},
 12_{a392}
\end{array}
 \\ \hline 
 q^4 \left(3 q^6+14 q^4+25 q^2+22\right) & 
\begin{array}{c}
 12_{a345},
 12_{a356} 
\end{array}
 \\ \hline 
 q^6 \left(5 q^4+28 q^2+31\right) & 
\begin{array}{c}
 12_{a431} ,
 12_{a659} ,
 12_{a900} ,
 12_{a973} ,\\
 12_{a995} ,
 12_{a1004} ,
 12_{a1112} 
\end{array}
 \\ \hline 
 q^4 \left(5 q^6+25 q^4+50 q^2+41\right) & 
\begin{array}{c}
 12_{a610},
 12_{a653} 
\end{array}
 \\ \hline 
 3 q^4 \left(q^2+2\right) \left(q^2+3\right)^2 & 
\begin{array}{c}
 12_{a880}
\end{array}
 \\ \hline 
 q^4 \left(5 q^6+32 q^4+73 q^2+59\right) & 
\begin{array}{c}
 12_{a974} 
\end{array}
 \\ \hline 
 q^4 \left(5 q^6+38 q^4+86 q^2+67\right) & 
\begin{array}{c}
 12_{a996} 
\end{array}
 \\ \hline 
 q^4 \left(q^2+3\right) \left(5 q^4+21 q^2+23\right) & 
\begin{array}{c}
 12_{a1037} 
\end{array}
 \\ \hline 
 q^4 \left(7 q^6+52 q^4+113 q^2+84\right) & 
\begin{array}{c}
 12_{a1097} 
\end{array}
 \\ \hline 
 q^4 \left(5 q^6+43 q^4+100 q^2+77\right) & 
\begin{array}{c}
 12_{a1113} 
\end{array}
 \\ \hline 
 9 q^4 \left(q^2+2\right)^2 & 
\begin{array}{c}
 12_{n881} 
\end{array}
 \\
\end{array}
\]
    \caption{Positive prime knots with 10-12 crossings such that its Links-Gould invariant has top coefficient that is not a product of $a_n(q)$'s.}
    \label{table: positive knots, non-annular}
  \end{table}

\section{Additional comments}
\label{section: additional comments}

In this section we discuss the following question, which motivated our study of links bounding annular surfaces: is there an analogue of the Giroux-Goodman theorem for strongly quasi-positive links? More precisely\footnote{Note that in the case of fiber surfaces/fibred links, both questions are equivalent (since the fiber surface is the unique minimal genus Seifert surface of a fibred link) and they are answered by the theorem of Giroux-Goodman.}:
\begin{enumerate}
    \item Is there a simple set of QP surfaces from which all other QP surfaces are obtained by plumbing/deplumbing?
    \item Is there a simple set of QP surfaces from which all SQP links bound a surface made by plumbing/deplumbing surfaces in such a set?
\end{enumerate}
\medskip

Recall that a surface $\S$ embedded in $S^3$ is called {\em quasi-positive} (QP) if it embeds as an incompressible subsurface of the fiber surface $F_T$ of a positive torus link $T$. Here, incompressible means that $\pi_1(\S)\to\pi_1(F_T)$ is injective. This is equivalent to the original definition of Rudolph in terms of braids \cite{Rudolph:quasipositive3-characterization}. A link is {\em strongly quasi-positive} (SQP) if it bounds a quasi-positive surface. The above questions are motivated by the fact that being a QP surface is stable under plumbing and deplumbing (or more generally, Murasugi sum) \cite{Rudolph:quasipositive-plumbing}. 
\medskip

We had considered the set of positively twisted annuli $A_n,n>0$ as a candidate for the above two questions. We will say that a surface is {\em positive annular} if it is obtained by plumbing/deplumbing annuli $A_n$ with $n>0$, positive annular surfaces are QP by \cite{Rudolph:quasipositive-plumbing}. It is easy to see that not all QP surfaces are positive annular: indeed, an easy Seifert van Kampen argument shows that positive annular surfaces have free fundamental group, which is not the case for general QP surfaces. The following example was provided to us by Sebastian Baader: for every knot $K$ there is a QP annulus $\S_K$ having that knot as core, hence if the knot is non-trivial then $\pi_1(S^3\sm \S_K)$ is not free. 
\medskip

The question of whether any SQP link bounds a positive annular surface is a bit more subtle than the previous one. Indeed, if a surface is QP but not annular (as in Baader's example), it could still be possible that its boundary bounds another minimal genus Seifert surface that is positive annular. Note that non-fibred links may have several non-isotopic minimal genus Seifert surfaces (indeed, by \cite[Corollary 3.2]{Gabai:fibred}, the result of Murasugi summing two non-fibred surfaces always results in a link with more than one minimal genus Seifert surface). 
\medskip

Still, we answer this in the negative, even for the smaller class of positive links. Note that every positive link is SQP \cite{Rudolp:positive-are-SQ} and {\em all} their genus-minimizing Seifert surfaces have free fundamental group \cite{Ozawa:incompressible-positive}. As explained in Section \ref{section: computations}, a consequence of Corollary \ref{Corollary: obstruction to bound ANNULAR surface} is that the positive links of Table \ref{table: positive knots} do not bound annular surfaces at all. 

\medskip

\medskip

\medskip

\bibliographystyle{amsplain}
\bibliography{referencesabr.bib}

\end{document}